\documentclass[12pt]{article}

\usepackage[centertags]{amsmath}
\usepackage{amsfonts}
\usepackage{amssymb}
\usepackage{amsthm}
\usepackage{newlfont}

\newlength{\defbaselineskip}
\newcommand{\setlinespacing}[1]%
           {\setlength{\baselineskip}{#1 \defbaselineskip}}

\theoremstyle{plain}
\newtheorem{thm}{Theorem}[section]
\newtheorem{cor}[thm]{Corollary}
\newtheorem{lem}[thm]{Lemma}

\theoremstyle{definition}
\newtheorem{defn}[thm]{Definition}
\newtheorem{rem}[thm]{Remark}

\numberwithin{equation}{section}

\begin{document}

\newcommand{\ol }{\overline}
\newcommand{\ul }{\underline }
\newcommand{\ra }{\rightarrow }
\newcommand{\lra }{\longrightarrow }
\newcommand{\ga }{\gamma }
\newcommand{\st }{\stackrel }
\newcommand{\scr }{\scriptsize }

\title{\Large\textbf{An Outer Commutator Multiplier and Capability of Finitely Generated Abelian Groups}\footnote{This research was in part supported by a grant from IPM NO.85200018}}
\author{\textbf{Mohsen Parvizi}\footnote{Corresponding author}\\ Department of Pure Mathematics,\\ Damghan University of Basic
Sciences,\\ P.O.Box 36715-364, Damghan, Iran, and \\ Institute for
Studies in Theoretical Physics\\ and Mathematics (IPM), P.O.Box 5746-19395, Tehran, Iran.\\ \textbf{Behrooz Mashayekhy} \\
Department of Mathematics,\\ Ferdowsi University of Mashhad,\\
P.O.Box 1159-91775, Mashhad, Iran.\\
 E-mail: parvizi@dubs.ac.ir \&
mashaf@math.um.ac.ir}
\date{ }
\maketitle
\begin{abstract}
We present an explicit structure for the Baer invariant of a
finitely generated abelian group with respect to the variety
$[\mathfrak{N}_{c_1},\mathfrak{N}_{c_2}]$, for all $c_2\leq c_1\leq
2c_2$. As a consequence we determine necessary and sufficient
conditions for such groups to be
$[\mathfrak{N}_{c_1},\mathfrak{N}_{c_2}]$-capable. We also show that
if $c_1\neq 1\neq c_2$, then a finitely generated abelian group is
$[\mathfrak{N}_{c_1},\mathfrak{N}_{c_2}]$-capable if and only if it
is capable. Finally we show that $\mathfrak{S}_2$-capability implies
capability but there is a finitely generated abelian group which is
capable but is not ${\mathfrak S}_2$-capable.
\end{abstract}
\textit{Key Words}: Baer invariant; Finitely generated abelian
group; Varietal capability; Outer commutator variety.\\
\textit{2000 Mathematics Subject Classification}: 20E34; 20E10; 20F12; 20F18.\\

\section{Introduction and Preliminaries}

An interesting problem connected to the notion of Baer invariants is
the computation of Baer invariants for some natural classes of
groups with respect to common varieties. The class of finitely
generated abelian groups is an appropriate candidate because of
their explicit structure theorem.

First of all Schur (1907), computed the Schur multiplier of a finite
abelian group. The second author in a joint paper Mashayekhy and
Moghaddam (1997), computed the $\mathfrak{N}_c$-multiplier of
finitely generated abelian groups, where $\mathfrak{N}_c$ is the
variety of all nilpotent groups of class at most $c$. The authors in
2006 (Mashayekhy and Parvizi, 2006), computed the polynilpotent
multipliers of finitely generated abelian groups.

Another interesting problem is determining capable groups or more
generally varietal capable groups. In 1938 Baer classified all
capable groups among the direct sums of cyclic groups and in
particular among the finitely generated abelian groups. Burns and
Ellis  (1998), extended the result for $\mathfrak{N}_c$-capability
and recently the authors in a joint paper (Parvizi, et al.) with S.
Kayvanfar classified all finitely generated abelian groups that are
polynilpotent capable. Some work has been done in other classes of
groups for example Magidin (2005), worked on capability of the
nilpotent product of cyclic groups.

We note that one reason for studying Baer invariants and varietal
capability is their relevance to the isologism theory of P. Hall
which is used to classify groups such as prime-power groups into a
suitable equivalence classes coarser than isomorphism. The article
of Leedham-Green and Mckay (1976), gives a fairly comprehensive
account of these relationships.

In this paper we compute the multiplier of finitely generated
abelian groups and determine all varietal capable finitely generated
abelian groups with respect to the variety
$[\mathfrak{N}_{c_1},\mathfrak{N}_{c_2}],$ for all $c_2\leq c_1\leq
2c_2.$

In particular we show that:

$a)$ if $c_1\neq 1\neq c_2$, then a finitely generated abelian group
is $[\mathfrak{N}_{c_1},\mathfrak{N}_{c_2}]$-capable if and only if
it is capable;

$b)$ every $\mathcal{S}_2$-capable group is a capable group and
there is a finitely generated abelian group which is capable but is
not $\mathcal{S}_2$-capable.

In the following there are some preliminaries which are needed.

\begin{defn}
Let $G$ be any group with a free presentation $G\cong F/R$, where
$F$ is a free group. Then, after Baer (1945), the \textit{Baer
invariant} of $G$ with respect to a variety of groups
$\mathfrak{V}$, denoted by $\mathfrak{V}M(G)$, is defined to be
$$\mathfrak{V}M(G)=\frac{R\cap V(F)}{[RV^*F]}\ ,$$
where $V$ is the set of words of the variety $\mathfrak{V}$, $V(F)$
is the verbal subgroup of $F$ with respect to $\mathfrak{V}$ and
\begin{eqnarray*}
\lefteqn{[RV^*F] = \Bigl\langle
v(f_1,\ldots,f_{i-1},f_ir,f_{i+1},\ldots,f_n)
v(f_1,\ldots,f_i,\ldots,f_n)^{-1} \,\Bigm|}\hspace{1.5truein}\\
&&r\in R, v\in V, f_i\in F\mbox{\ for all\ }1\leq i\leq n,
n\in\mathbf{N}\Bigr\rangle.
\end{eqnarray*}

As a special case, if $\mathfrak{V}$ is the variety of abelian
groups, $\mathfrak{A}$, the Baer invariant of $G$ is the well-known
\textit{Schur multiplier}
$$\frac{R\cap F'}{[R,F]}.$$

If $\mathfrak{N}_c$ is the variety of nilpotent groups of class at
most $c\geq1$, then the Baer invariant of $G$ with respect to it, is
called the \textit{$c$-nilpotent multiplier} of $G$, is given by:
$$\mathfrak{N}_cM(G)=\frac{R\cap \gamma_{c+1}(F)}{[R,\ _cF]},$$
where $\gamma_{c+1}(F)$ is the $(c+1)$-st term of the lower central
series of $F$ and $[R,\ _1F]=[R,F], [R,\ _cF]=[[R,\ _{c-1}F],F]$,
inductively.
\end{defn}

\begin{lem}
(Hulse and Lennox 1976)\textit{ If $u$ and $w$ are any two words and
$v=[u,w]$ and $K$ is a normal subgroup of a group $G$, then}
$$ [Kv^*G]=[[Ku^*G],w(G)][u(G),[Kw^*G]].$$
\end{lem}

\begin{proof}
See Hall and Senior (1964, Lemma 2.9).
\end{proof}

 Now, using the above lemma, then
 the Baer invariant of a group $G$ with respect to the outer commutator variety $[\mathfrak{N}_{c_1},\mathfrak{N}_{c_2}]$, is as follows:
 $$[\mathfrak{N}_{c_1},\mathfrak{N}_{c_2}]M(G)\cong \frac{R\cap
 [\ga_{c_1+1}(F),\ga_{c_2+1}(F)]}{[R,\ _{c_1}F,\ga_{c_2+1}(F)][R,\ _{c_2}F,\ga_{c_1+1}(F)]}.$$
\begin{defn}\label{basic}
Let $X$ be an independent subset of a free group, and select an
arbitrary total order for $X$. We define the basic commutators on
$X$, their weight \textit{wt}, and the ordering among them as
follows:

(1) \ The elements of $X$ are basic commutators of weight one,
ordered according to the total order previously chosen.

(2) \ Having defined the basic commutators of weight less than $n$,
the basic commutators of weight $n$ are the $c_k=[c_i,c_j]$, where:

(a) \ $c_i$ and $c_j$ are basic commutators and $wt(c_i)+wt(c_j)=n$,
and

(b) \ $c_i>c_j$, and if $c_i=[c_s,c_t]$ then $c_j\geq c_t$.

(3) \ The basic commutators of weight $n$ follow those of weight
less than $n$. The basic commutators of weight $n$ are ordered among
themselves lexicographically; that is, if $[b_1,a_1]$ and
$[b_2,a_2]$ are basic commutators of weight $n$, then $[b_1,a_1]\leq
[b_2,a_2]$ if and only if $b_1<b_2$ or $b_1=b_2$ and $a_1<a_2$.
\end{defn}

The next two theorems are vital in our investigation.
\begin{thm}\label{Hall}
(Hall, 1959). Let $F=\langle x_1,x_2,\ldots ,x_d\rangle $ be a free
group, then
$$ \frac {\ga_n(F)}{\ga_{n+i}(F)} \ \ , \ \ \ \  1\leq i\leq n$$
is the free abelian group freely generated by the basic commutators
of weights $n,n+1,\ldots ,n+i-1$ on the letters $\{x_1,\ldots
,x_d\}.$
\end{thm}

\begin{thm}
(Witt Formula). The number of basic commutators of weight $n$ on $d$
generators is given by the following formula:
$$ \chi_n(d)=\frac {1}{n} \sum_{m|n}^{} \mu (m)d^{n/m},$$
where $\mu (m)$ is the M\"{o}bius function, which is defined to be
   \[ \mu (m)=\left \{ \begin{array}{ll}
      1 & \textrm{if}\ \ m=1, \\ 0 &  \textrm{if} \ \ m=p_1^{\alpha_1}\ldots
p_k^{\alpha_k}\ \ \exists \alpha_i>1, \\ (-1)^s & \textrm{if} \ \
m=p_1\ldots p_s,
\end{array} \right.  \]
where the $p_i$, $1\leq i\leq k$, are the distinct primes dividing
$m$.
\end{thm}
\begin{proof}
See Hall (1959).
\end{proof}
The following definition will be used several times in this article.

\begin{defn}\label{cap}
Let $\mathfrak{V}$ be any variety of groups defined by a set of laws
$V$, and $G$ be any group. Extending the terminology of Hall and
Senior (1964), $G$ is called ${\cal V}$-capable if there exists a
group $E$ which satisfies $G\cong E/V^*(E)$, where $V^*(E)$ is the
marginal subgroup of $E$ with respect to $\mathfrak{V}$. (See also
Moghaddam and Kayvanfar, 1997, for the definition of
$\mathfrak{V}$-capability and Burns and Ellis, 1998, for
$\mathfrak{N}_c$-capability.)
\end{defn}

According to Definition~\ref{cap} capable groups are
$\mathfrak{A}$-capable groups, where $\mathfrak{A}$ is the variety
of abelian groups.

The following definition and theorem are taken from Moghaddam and
Kayvanfar (1997), and contains a necessary and sufficient condition
for a group to be $\mathfrak{V}$-capable.

\begin{defn}
Let $\mathfrak{V}$ be any variety and $G$ be any group. Define
$V^{**}(G)$ as follows:
$$V^{**}(G)=\cap \{\psi(V^*(E))\ | \ \psi:E \st{onto}\lra G \ , \ ker\psi\subseteq V^*(E)\}.$$
Note that if $\mathfrak{V}$ is the variety of abelian groups, then
the above notion which has been first studied in Beyl., et al.
(1979), is denoted by $Z^*(G)$ and called epicenter in Burns and
Ellis (1998). Also the above notion has been studied in Burns and
Ellis (1998), for the variety $\mathfrak{N}_c$.
\end{defn}

\begin{thm}
With the above notations and assumptions  $G/V^{**}(G)$ is the
largest quotient of $G$ which is $\mathfrak{V}$-capable, and hence
$G$ is $\mathfrak{V}$-capable if and only if $V^{**}(G)=1$.
\end{thm}

The following theorem and its conclusion state the relationship
between $\mathfrak{V}$-capability and Baer invariants.

\begin{thm}\label{inj}
Let $\mathfrak{V}$ be any variety of groups, $G$ be any group, and
$N$ be a normal subgroup of $G$ contained in the marginal subgroup
with respect to $\mathfrak{V}$. Then the natural homomorphism
$\mathfrak{V}M(G)\longrightarrow \mathfrak{V}M(G/N)$ is injective if
and only if $N\subseteq V^{**}(G)$, where $\mathfrak{V}M(G)$ is the
Baer invariant of $G$ with respect to $\mathfrak{V}$.
\end{thm}

\begin{proof}
See Moghaddam and Kayvanfar (1997).
\end{proof}

In the finite case the following theorem is easier to use than the
proceeding ones.

\begin{thm}
Let $\mathfrak{V}$ be any variety and $G$ be any group with
$V(G)=1$. If $\mathfrak{V}M(G)$ is finite, and $N$ is a normal
subgroup of $G$ such that $\mathfrak{V}M(G/N)$ is also finite, then
the natural homomorphism $\mathfrak{V}M(G)\longrightarrow
\mathfrak{V}M(G/N)$ is injective if and only if
$|\mathfrak{V}M(G/N)|=|\mathfrak{V}M(G)|$.
\end{thm}

\begin{proof}
It is easy to see that with the assumption of the theorem we have
$\mathfrak{V}M(G)\cong V(F)/[RV^*F]$ and $\mathfrak{V}M(G/N)\cong
V(F)/[SV^*F]$ in which $G\cong F/R$ is a free presentation for $G$
and $N\cong S/R$. Therefore the kernel of the natural homomorphism
$\mathfrak{V}M(G)\longrightarrow \mathfrak{V}M(G/N)$ is the group
$[SV^*F]/[RV^*F]$. Considering the finiteness of $\mathfrak{V}M(G)$
and $\mathfrak{V}M(G/N)$, the result easily follows.
\end{proof}

As a useful consequence of Theorem~\ref{inj} we have:

\begin{cor}
An abelian group $G$ is $\mathfrak{V}$-capable if and only if the
natural homomorphism $\mathfrak{V}M(G)\longrightarrow
\mathfrak{V}M(G/\langle x\rangle)$ has a non-trivial kernel for all
non-identity elements $x$ in $V^*(G)$.
\end{cor}

The following fact is used in the last section [Stroud, (1965),
Theorem 1.2(b)]

\begin{thm}
Let $u$ and $v$ be two words in independent variables and $w=[u,v]$.
Then, in any group $G$,\\
$(i)$ \ $w(G)=[u(G),v(G)]$\\
$(ii)$ \ if $A=C_G(u(G))$, $B=C_G(v(G))$, $L/A=v^*(G/A)$, and
$M/B=u^*(G/B)$, then $w^*(G)=L\cup M$.
\end{thm}

To use these results we need an explicit structure for the Baer
invariants of finitely generated abelian groups with respect to the
variety $\mathfrak{V}$ as defined. This will be done in Theorem 2.6.

\section{Computing $[\mathfrak{N}_{c_1},\mathfrak{N}_{c_2}]$-Multipliers }

Let $G \cong \mathbb{Z}^{(k)}\oplus \mathbb{Z}_{n_1}\oplus
\mathbb{Z}_{n_2}\oplus\cdots\oplus\mathbb{Z}_{n_t}$ be a finitely
generated abelian group with $n_{i+1}\mid n_i$ for all $1\leq i\leq
t-1$,  where for any group $X$, $X^{(n)}$ denotes the group $X\oplus
X\oplus \cdots \oplus X$ ($n$ copies). Let $F=F\langle
x_1,\ldots,x_k,x_{k+1},\ldots,x_{k+t}\rangle$ be the free group on
the set $\{x_1,\ldots,x_{k+t}\}$. It is easy to see that
$$1\lra R\lra F\lra G\lra 1,$$ is a free presentation for $G$ in
which $R=\prod_{i=1}^{t}R_i\ga_2(F)$, where $R_i=\langle
x_{k+i}^{n_i}\rangle$, so the Baer invariant of $G$ with respect to
$[\mathfrak{N}_{c_1},\mathfrak{N}_{c_2}]$ is

$$[\mathfrak{N}_{c_1},\mathfrak{N}_{c_2}]M(G)\cong \frac{R\cap
 [\ga_{c_1+1}(F),\ga_{c_2+1}(F)]}{[R,\ _{c_1}F,\ga_{c_2+1}(F)][R,\ _{c_2}F,\ga_{c_1+1}(F)]}.$$
Since $R\supseteq \ga_2(F)$ we have
$$[\mathfrak{N}_{c_1},\mathfrak{N}_{c_2}]M(G)\cong \frac{[\ga_{c_1+1}(F),\ga_{c_2+1}(F)]}{[R,\ _{c_1}F,\ga_{c_2+1}(F)][R,\ _{c_2}F,\ga_{c_1+1}(F)]}.$$

 In order to find the structure of $[\mathfrak{N}_{c_1},\mathfrak{N}_{c_2}]M(G)$, we
need the following notation and lemmas. Using
Definition~\ref{basic}, we define the following set when $c_1\geq
c_2$.


\begin{center}
$A$=\{$[\beta,\alpha] \ | \ \beta $ and $ \alpha $ are basic
commutators on $X$ such that $ \beta>\alpha, $ $ \ wt(\beta)=c_1+1, \ wt(\alpha)=c_2+1$ \}.\\ \ \  \\

\end{center}

\begin{lem}
If $c_1\leq 2c_2$, then every element of $A$ is a basic commutator
on~$X$.
\end{lem}

\begin{proof}
Every element of $A$ has the form $[\beta,\alpha]$, where $\beta$
and $\alpha$ are basic commutators on $X$, $\beta>\alpha$ and
$wt(\beta)=c_1+1, wt(\alpha)=c_2+1$. Now, let
$\beta=[\beta_1,\beta_2]$, then in order to show that
$[\beta,\alpha]$ is a basic commutator on $X$, it is enough to show
that $\beta_2\leq \alpha$. Since $\beta=[\beta_1,\beta_2]$ is a
basic commutator on $X$, $\beta_1>\beta_2 $ and hence
$wt(\beta_2)\leq \frac{1}{2}wt(\beta)$. Now, if $c_1\leq 2c_2$, then
$\frac{1}{2}(c_1+1)<c_2+1$. Thus, since $wt(\beta)=c_1+1$, we have
$$wt(\beta_2)\leq \frac{1}{2}wt(\beta)=\frac{1}{2}(c_1+1)<c_2+1=wt(\alpha).$$
Therefore $\beta_2<\alpha$ and hence the result holds.
\end{proof}

Now put $H=[R,\ _{c_1}F,\ga_{c_2+1}(F)][R,\
_{c_2}F,\ga_{c_1+1}(F)]\cap \ga_{c_1+c_2+3}(F)$ we have the
following.

\begin{lem}
$[\ga_{c_1+1}(F),\ga_{c_2+1}(F)]\equiv \langle A\rangle \pmod{H}$.
\end{lem}

\begin{proof}
Let $[\beta,\alpha]$ be a generator of the group
$[\ga_{c_1+1}(F),\ga_{c_2+1}(F)]$, so we have $\beta\in
\ga_{c_1+1}(F)$ and $\alpha\in \ga_{c_2+1}(F)$. Now by
Theorem~\ref{Hall} we can write
$\beta=\beta_1\beta_2\ldots\beta_r\eta$ and
$\alpha=\alpha_1\alpha_2\ldots\alpha_s\mu$ in which the $\beta_j$
are basic commutators on $X$ of weight $c_1+1$, the $\alpha_i$ are
basic commutators on $X$ of weight $c_2+1$, $\eta\in\ga_{c_1+2}(F)$
and $\mu\in\ga_{c_2+2}(F)$. Now
$[\beta,\alpha]$
will be a product of factors of the forms
$[\beta_j,\alpha_i]^{f_{ij}}$, $[\beta_j,\mu]^{g_j}$,
$[\eta,\alpha_i]^{h_i}$ and $[\eta,\mu]^k$, in which
$f_{ij},g_j,h_i,k\in \ga_{c_2+1}(F)$. Now by the Three Subgroup
Lemma  it is easy to see that $[\beta_j,\alpha_i,f_{ij}],
[\beta_j,\mu]^{g_j}, [\eta,\alpha_i]^{h_i}$ and $[\eta,\mu]^k\in H$.
Hence the result holds.
\end{proof}

Now the group $[\ga_{c_1+1}(F),\ga_{c_2+1}(F)]/H$ is the group
generated by the set $\bar{A}=\{aH\ \mid \ a\in A\}$. The following
shows that it is in fact the free abelian group with the basis
$\bar{A}$.

\begin{lem}
With the above notation and assumptions
$[\ga_{c_1+1}(F),\ga_{c_2+1}(F)]/H$ is the free abelian group with
the basis $\bar{A}$.
\end{lem}

\begin{proof}
The group is abelian and generated by $\bar{A}$, which is the image
of the elements of $A$ modulo $H$. The elements of $A$ are basic
commutators of weight $c_1+c_2+2$ (Lemma 2.1), and hence linearly
independent over $\gamma_{c_1+c_2+3}(F)$; the latter contains $H$,
so the elements of $A$ are also linearly independent modulo $H$.
\end{proof}

Now by the isomorphism
$$[\mathfrak{N}_{c_1},\mathfrak{N}_{c_2}]M(G)\cong\frac{[\ga_{c_1+1}(F),\ga_{c_2+1}(F)]/H}{[R,\ _{c_1}F,\ga_{c_2+1}(F)][R,\ _{c_2}F,\ga_{c_1+1}(F)]/H},$$
in order to determine the explicit structure of
$[\mathfrak{N}_{c_1},\mathfrak{N}_{c_2}]M(G)$ we only need to
determine the structure of $[R,\ _{c_1}F,\ga_{c_2+1}(F)][R,\
_{c_2}F,\ga_{c_1+1}(F)]/H$. We actually show that the mentioned
group is free abelian with basis $\cup_{i=1}^{t}\bar{B_i}$ where the
$B_i$ consist of $n_i$th powers of suitable elements of $\bar{A}$.
To do this we need the following lemma.

\begin{lem}
With the previous notation we have
$$[R,\ _{c_1}F,\ga_{c_2+1}(F)]\equiv\prod[R_i,\ _{c_1}F,\ga_{c_2+1}(F)] \pmod{H}$$ and
$$[R,\ _{c_2}F,\ga_{c_1+1}(F)]\equiv\prod[R_i,\ _{c_2}F,\ga_{c_1+1}(F)]\pmod{H}$$
\end{lem}

\begin{proof}
By routine commutator calculus we have
\begin{eqnarray*}
[R,{}_{c-1}F,\gamma_{c_2+1}(F)] & = & \prod_{i=1}^t
[R_i^F\gamma_2(F), {}_{c-1}F, \gamma_{c_2+1}(F)]\\
& = & \prod_{i=1}^t [R_i^F, {}_{c-1}F, \gamma_{c_2+1}(F)]
[\gamma_2(F), {}_{c-1}F, \gamma_{c_2+1}(F)]\\
&\equiv & \prod_{i=1}^t [R_i, {}_{c-1}F, \gamma_{c_2+1}(F)]\pmod{H}.
\end{eqnarray*}
\end{proof}
The following lemma will do most of the work.

\begin{lem}\label{main}
$[R,\ _{c_1}F,\ga_{c_2+1}(F)][R,\ _{c_2}F,\ga_{c_1+1}(F)]/H$ is the
free abelian group with the basis $\cup_{j=1}^{t}\bar{B_j}$, where
\begin{center}
$B_j=\{[\beta,\alpha]^{n_j} \ | \ [\beta,\alpha]\in A$ and $x_{k+j}$
does occur in $[\beta,\alpha]\}$ and \ \ $\bar{}$ \ \ denotes the
natural homomorphism $[\ga_{c_1+1}(F),\ga_{c_2+1}(F)]\longrightarrow
[\ga_{c_1+1}(F),\ga_{c_2+1}(F)]/H$.
\end{center}
\end{lem}

\begin{proof}
Clearly $[R,\ _{c_1}F]\equiv\prod_{i=k+1}^{k+t}[R_i,\ _{c_1}F] \pmod{\ga_{c_1+2}(F)}$.\\
Also
\begin{center}
$[R_i,\ _{c_1}F]\equiv\langle\beta^{n_i} \mid \beta$ is a basic
commutator of weight $c_1+1$ on $X$ s.t. $x_{k+i}$ does appear in
it$\rangle \ \pmod{\ga_{c_1+2}(F)}$.
\end{center}
Therefore
\begin{center}
$[R_i,\ _{c_1}F,\ga_{c_2+1}(F)]\equiv\langle[\beta,\alpha]^{n_i}\mid
\ [\beta,\alpha]\in A$ and $x_{k+i}$ does appear in  $\beta \
\rangle \ \pmod{H}$.
\end{center}
Similarly,
\begin{center}
$[R_i,\ _{c_2}F,\ga_{c_1+1}(F)]\equiv\langle[\beta,\alpha]^{n_i}\mid
\ [\beta,\alpha]\in A$ and $x_{k+i}$ does appear in   $\alpha
\rangle \pmod{H}$.
\end{center}
Hence $[R,\ _{c_1}F,\ga_{c_2+1}(F)][R,\
_{c_2}F,\ga_{c_1+1}(F)]\equiv\langle\bigcup B_j\rangle \pmod{H}.$

\end{proof}
Now we are in a position to give an explicit structure for
$[\mathfrak{N}_{c_1},\mathfrak{N}_{c_2}]M(G)$. It only remains to
compute $|B_j|$ for $j=1,...,t$. Bearing in mind Lemma 2.3 it is
clear that each $\bar{B_j}$ is linearly independent modulo $H$
therefore the size of $\bar{B_j}$ is as same as that of $B_j$, so it
is enough to compute the size of $B_j$. The cases $c_1=c_2$ and
$c_1>c_2$ are essentially different in computing $|B_j|$. If
$c_1>c_2$ for an arbitrary $j$ we can write $B_j=B_{1j}\cup
B_{2j}\cup B_{3j}$ in which

$B_{1j}=\{[\beta,\alpha]^{n_j} \ | \ [\beta,\alpha]^{n_j}\in B_j$
and $x_{k+j}$ only appears in $\beta\}$,

$B_{2j}=\{[\beta,\alpha]^{n_j} \ | \ [\beta,\alpha]^{n_j}\in B_j$
and $x_{k+j}$ only appears in $\alpha\}$,

$B_{3j}=\{[\beta,\alpha]^{n_j} \ | \ [\beta,\alpha]^{n_j}\in B_j$
and
$x_{k+j}$ appears in both $\beta$ and $\alpha\}$.\\
It is easy to see that the union is disjoint and we have

$|B_{1j}|=(\chi_{c_1+1}(k+j)-\chi_{c_1+1}(k+j-1))\chi_{c_2+1}(k+j-1),$

$|B_{2j}|=\chi_{c_1+1}(k+j-1)(\chi_{c_2+1}(k+j)-\chi_{c_2+1}(k+j-1)),$\\
and

$|B_{3j}|=(\chi_{c_1+1}(k+j)-\chi_{c_1+1}(k+j-1))(\chi_{c_2+1}(k+j)-\chi_{c_2+1}(k+j-1)),$
so
$|B_j|=\chi_{c_1+1}(k+j)\chi_{c_2+1}(k+j)-\chi_{c_1+1}(k+j-1)\chi_{c_2+1}(k+j-1)$.

In the case $c_1=c_2$ it is easy to see that $A$ is in fact the set
of all basic commutators of weight 2 on the set of all basic
commutators of weight $c_1$, so we have
$|B_j|=\chi_2(\chi_{c_1+1}(k+j))-\chi_2(\chi_{c_1+1}(k+j-1))$.

Now the following theorem gives the desired structure.

\begin{thm}
Let $G \cong \mathbf{Z}^{(k)}\oplus \mathbb{Z}_{n_1}\oplus
\mathbb{Z}_{n_2}\oplus\cdots\oplus\mathbb{Z}_{n_t}$ be a finitely
generated abelian group with $n_{i+1}\mid n_i$ for all $1\leq i\leq
t-1$, if $c_2\leq c_1\leq 2c_2$ then,
$$[\mathfrak{N}_{c_1},\mathfrak{N}_{c_2}]M(G)\cong
\mathbb{Z}^{(b_{k})}\oplus \mathbb{Z}_{n_1}^{(b_{k+1}-b_k)}\oplus
\mathbb{Z}_{n_2}^{(b_{k+2}-b_{k+1})}\oplus\ldots\oplus\mathbb{Z}_{n_t}^{(b_{k+t}-b_{k+t-1})}$$
where $b_i=\chi_{c_1+1}(i)\chi_{c_2+1}(i)$, if $c_1>c_2$ and
$b_i=\chi_2(\chi_{c_1+1}(i))$ if $c_1=c_2$.
\end{thm}

\begin{proof}
The proof easily follows from Lemma~\ref{main}.
\end{proof}

Comparing this theorem with the main theorem of Mashayekhy and
Parvizi (2006), it is easy to see that they agree on the variety
$\mathfrak{N}_{c,1}$ in the formula for the Baer invariant.

\section{$[\mathfrak{N}_{c_1},\mathfrak{N}_{c_2}]$-Capability }

The concept of capable groups occured in work of P. Hall for
classifying $p$-groups. Determining such groups is an interesting
problem to study. Some researches has done on it which for example
see Burns and Ellis (1998), Ellis (1996), Moghaddam and Kayvanfar
(1997), and Magidin (2005). In this section we explicitly determine
the structure of all
$[\mathfrak{N}_{c_1},\mathfrak{N}_{c_2}]$-capable groups in the
class of finitely generated abelian groups. When $c_2\leq c_1\leq
2c_c$ to do this we wish to use Theorems 1.9, 1.10, and Corollary
1.11. To use them the structure of the subgroups of a finitely
generated abelian group is needed as well as the structure of the
Baer invariant of $G$. Theorem 2.6 gives the latter and the
following will determine the structure of the desired subgroups.

\begin{lem}
Let $G$ be a finitely generated abelian group and $H\leq G$. Then
$r_0(G)=r_0(G/H)+r_0(H)$, where $r_0(X)$ is the torsion free rank of
a finitely generated abelian group $X$.
\end{lem}

\begin{proof}
See Fuchs (1970).
\end{proof}

In the case of $p$-groups the following theorem has an important
role in our investigation.

\begin{thm}\label{pgroup}
Let $G\cong \mathbb{Z}_{p^{\alpha_1}}\oplus
\mathbb{Z}_{p^{\alpha_2}}\oplus\cdots\oplus
\mathbb{Z}_{p^{\alpha_k}}$ be a finite abelian $p$-group, where
$\alpha_{i+1}\leq \alpha_i$ for all $1\leq i \leq k-1$, and let $H$
be a subgroup of $G$. Then $H\cong \mathbb{Z}_{p^{\beta_1}}\oplus
\mathbb{Z}_{p^{\beta_2}}\oplus\cdots\oplus \mathbb{Z}_{p^{\beta_k}}$
where $\beta_{i+1}\leq \beta_i$ for all $1\leq i \leq k-1$ and
$0\leq \beta_i\leq \alpha_i$ for $1\leq i \leq k$.
\end{thm}

\begin{proof}
See Fuchs (1970).
\end{proof}

\begin{thm}
Let $G\cong \mathbb{Z}^{(k)}\oplus
\mathbb{Z}_{n_1}\oplus\cdots\oplus \mathbb{Z}_{n_t}$ be a finitely
generated abelian group, where $n_{i+1}\mid n_i$ for $1\leq i \leq
t-1$, and let $H$ be a finite subgroup of $G$. Then $H\cong
\mathbb{Z}_{m_1}\oplus\cdots\oplus \mathbb{Z}_{m_t}$, where
$m_{i+1}\mid m_i$ for all $1\leq i \leq t-1$ and $m_i\mid n_i$ for
all $1\leq i \leq t$.
\end{thm}

\begin{proof}
Trivially $H\leq t(G)$, the maximal torsion subgroup of $G$, so
without loss of generality we may assume that $G$ is finite. It is
well known that $G\cong S_{p_1}\oplus \cdots\oplus S_{p_t}$, where
$S_{p_i}$ is the $p_i$-Sylow subgroup of $G$. One may easily show
that if $H\cong S'_{p_1}\oplus \cdots\oplus S'_{p_t}$ is the same
decomposition for $H$, then $S'_{p_i}\leq S_{p_i}$ for all $1\leq i
\leq t$. Therefore it is enough to consider finite abelian
$p$-groups. Now Theorem~\ref{pgroup} completes the proof.
\end{proof}

Proceeding now to
$[\mathfrak{N}_{c_1},\mathfrak{N}_{c_2}]$-capability, note that
$[\mathfrak{N}_1,\mathfrak{N}_1]=\mathfrak{S}_2$ is the variety of
metabelian groups (that is groups of solvability length at most 2)
and, according to Theorem 2.6, $\mathfrak{S}_2M(G)=0$ whenever $G$
has at most two generators. But if $c_2<c_1\leq 2c_2$ or
$c_1=c_2>1$, then the Baer invariant is trivial only if $G$ is
cyclic. This suggests dealing with the two cases separately, and so
we assume first that $c_2<c_1\leq 2c_2$ or $c_1=c_2>1$.

The method we use here implies separating the cases which $G$ is
finite or infinite.

Case one: $G$ is finite abelian group.

\begin{thm}
Let $G\cong \mathbb{Z}_{n_1}\oplus\cdots\oplus \mathbb{Z}_{n_t}$ be
a finite abelian group, where $n_{i+1}\mid n_i$ for $1\leq i \leq
t-1$, then $G$ is $[\mathfrak{N}_{c_1},\mathfrak{N}_{c_2}]$-capable
if and only if $t\geq 2$ and $n_1=n_2$.
\end{thm}

\begin{proof}
We will establish the necessity by contrapositive. If $t=1$ then $G$
and all its quotients are cyclic abelian groups so by Theorem 2.6
$[\mathfrak{N}_{c_1},\mathfrak{N}_{c_2}]M(G/N)=0$ for any normal
subgroup $N$ of $G$, hence by Corollary 1.11 $G$ is not
$[\mathfrak{N}_{c_1},\mathfrak{N}_{c_2}]$-capable. On the other hand
if $n_1\neq n_2$, then let $x=({\bar n_2},{\bar 0},\ldots,{\bar
0})$, since $G/\langle x \rangle\cong \mathbb{Z}_{n_2}\oplus
\mathbb{Z}_{n_2}\cdots\oplus \mathbb{Z}_{n_t}$, Theorem 2.6 shows
the Baer invariants for $G$ and $G/\langle x \rangle$ have the same
size. This shows $G$ is not
$[\mathfrak{N}_{c_1},\mathfrak{N}_{c_2}]$-capable in this case by
Corollary 1.11.

For sufficiency, assume $t\geq 2$ and $n_1=n_2$. By Corollary 1.11
it is enough to show that if $N<G$ and
$[\mathfrak{N}_{c_1},\mathfrak{N}_{c_2}]M(G)\lra
[\mathfrak{N}_{c_1},\mathfrak{N}_{c_2}]M(G/N)$ is injective, then
$N$ is trivial.

In finite abelian groups each quotient is isomorphic to a subgroup
and vice versa. Now let $N<G$, then $G/N$ is isomorphic to a
subgroup of $G$, $H$ say; so by Theorem 3.3 $H\cong
\mathbb{Z}_{m_1}\oplus\cdots\oplus \mathbb{Z}_{m_t}$, where
$m_{i+1}\mid m_i$ for all $1\leq i \leq t-1$ and $m_i\mid n_i$ for
all $1\leq i \leq t$. Computing
$[\mathfrak{N}_{c_1},\mathfrak{N}_{c_2}]M(G)$ and
$[\mathfrak{N}_{c_1},\mathfrak{N}_{c_2}]M(H)$ using Theorem 2.6
shows that
$|[\mathfrak{N}_{c_1},\mathfrak{N}_{c_2}]M(G)|=|[\mathfrak{N}_{c_1},\mathfrak{N}_{c_2}]M(H)|$
if and only if $m_i=n_i$ for all $2\leq i \leq t$, but $n_1=n_2$ by
hypothesis which implies $n_1=m_1$ which is equivalent to $H=G$ and
hence $N=0$. Therefore $G$ is
$[\mathfrak{N}_{c_1},\mathfrak{N}_{c_2}]$-capable.
\end{proof}

Now we consider the infinite case.

\begin{thm}
Let $G\cong \mathbb{Z}^{(k)}\oplus
\mathbb{Z}_{n_1}\oplus\cdots\oplus \mathbb{Z}_{n_t}$ be an infinite
finitely generated abelian group, where $n_{i+1}\mid n_i$ for $1\leq
i \leq t-1$, then $G$ is
$[\mathfrak{N}_{c_1},\mathfrak{N}_{c_2}]$-capable, if and only if
$k\geq 2$.
\end{thm}

\begin{proof}
We first show that if $k=1$, then there exists a nontrivial element
$x$ of $G$ for which the natural homomorphism
$[\mathfrak{N}_{c_1},\mathfrak{N}_{c_2}]M(G)\longrightarrow
[\mathfrak{N}_{c_1},\mathfrak{N}_{c_2}]M(G/\langle x \rangle)$ is
injective, proving the necessity by contrapositive.

Suppose $k=1$, then $G\cong \mathbb{Z}\oplus
\mathbb{Z}_{n_1}\oplus\cdots\oplus \mathbb{Z}_{n_t}$. Let
$x=(n_1,{\bar 0},\ldots,{\bar 0})$, so $G/\langle x \rangle\cong
\mathbb{Z}_{n_1}\oplus \mathbb{Z}_{n_1}\oplus\cdots\oplus
\mathbb{Z}_{n_t}$. Now by Theorem 2.6 we have
$|[\mathfrak{N}_{c_1},\mathfrak{N}_{c_2}]M(G)|=|[\mathfrak{N}_{c_1},\mathfrak{N}_{c_2}]M(G/\langle
x \rangle)|$, so the result follows.

For sufficiency, assume that $k\geq 2$. It is enough to show that
there is no nontrivial subgroup $N$ of $G$ for which
$[\mathfrak{N}_{c_1},\mathfrak{N}_{c_2}]M(G)\lra
[\mathfrak{N}_{c_1},\mathfrak{N}_{c_2}]M(G/N)$ is injective. If $N$
is an infinite subgroup then $r_0(G/N)<r_0(G)$, so by Theorem 2.6
the torsion free rank of the Baer invariant of $G/N$ is strictly
smaller than that of the invariant for $G$, so no injection is
possible. On the other hand if $N$ is contained in the torsion
subgroup of $G$, so $G/N\cong \mathbb{Z}^{(k)}\oplus
\mathbb{Z}_{m_1}\oplus\cdots\oplus \mathbb{Z}_{m_t}$, where
$m_{i+1}\mid m_i$ and $m_i\mid n_i$ for all $1\leq i \leq t-1$, so
by Theorem 2.5 we have
$[\mathfrak{N}_{c_1},\mathfrak{N}_{c_2}]M(G)\cong
\mathbb{Z}^{(b_k)}\oplus
\mathbb{Z}_{n_1}^{(b_{k+1}-b_k)}\oplus\cdots\oplus
\mathbb{Z}_{n_t}^{(b_{k+t}-b_{k+t-1})}$ and\\
 $[\mathfrak{N}_{c_1},\mathfrak{N}_{c_2}]M(G/N)\cong \mathbb{Z}^{(b_k)}\oplus \mathbb{Z}_{m_1}^{(b_{k+1}-b_k)}\oplus\cdots\oplus \mathbb{Z}_{m_t}^{(b_{k+t}-b_{k+t-1})}$.
It is easy to show that
$$t([\mathfrak{N}_{c_1},\mathfrak{N}_{c_2}]M(G))=\mathbb{Z}_{n_1}^{(b_{k+1}-b_k)}\oplus\cdots\oplus
\mathbb{Z}_{n_t}^{(b_{k+t}-b_{k+t-1})}$$ and
$$t([\mathfrak{N}_{c_1},\mathfrak{N}_{c_2}]M(G/N))=\mathbb{Z}_{m_1}^{(b_{k+1}-b_k)}\oplus\cdots\oplus
\mathbb{Z}_{m_t}^{(b_{k+t}-b_{k+t-1})}.$$ The image of the torsion
subgroup of $[\mathfrak{N}_{c_1},\mathfrak{N}_{c_2}]M(G)$ under the
natural homomorphism must lie in the torsion subgroup of
$[\mathfrak{N}_{c_1},\mathfrak{N}_{c_2}]M(G/N)$, so if the
homomorphism is injective, then we must $t(G)=t(G/N)$;
$t(G/N)=t(G)/N$, this proves that if the map is ingective then
$N=0$, completing the proof.
\end{proof}

\begin{rem}
Let $G\cong \mathbb{Z}^{(k)}\oplus
\mathbb{Z}_{n_1}\oplus\cdots\oplus \mathbb{Z}_{n_t}$ be a finitely
generated abelian group, with $n_{i+1}\mid n_i$ for all $1\leq i
\leq t-1$. Baer's result Baer (1938), implies that $G$ is capable if
and only if $k\geq 2$ or $k=0$, $t\geq 2$ and $n_1=n_2$. Burns and
Ellis (1998), proved that $G$ is $\mathfrak{N}_c$-capable if and
only if it is capable. We now see that this also holds for
$[\mathfrak{N}_{c_1},\mathfrak{N}_{c_2}]$-capability with suitable
conditions on $c_1$ and $c_2$.
\end{rem}


In the case $c_1=c_2=1$ we only state the characterization of the
$\mathfrak S_2$-capable groups among finitely generated abelian
groups. The proofs are simillar to those of Theorems 3.4 and 3.5.
The needed lemmas and their proofs can be restated with necessary
changes similar to Theorems 3.4 and 3.5. Note that in this case the
variety $[\mathfrak{N}_{c_1},\mathfrak{N}_{c_2}]$ is actually the
variety of metabelian groups $\mathfrak{S}_2$.

\begin{thm}
Let $G\cong \mathbb{Z}^{(k)}\oplus
\mathbb{Z}_{n_1}\oplus\cdots\oplus \mathbb{Z}_{n_t}$ be a finitely
generated abelian group, where $n_{i+1}\mid n_i$ for all $1\leq i
\leq t-1$. Then $G$ is $\mathfrak{S}_2$-capable if and only if
$k\geq 3$, or $k=0$, $t\geq 3$, and $n_1=n_2=n_3$.
\end{thm}








\section{The relation between capability and $[\mathfrak{N}_{c_1},\mathfrak{N}_{c_2}]$-capability.}
As before mentioned in the beginning of section 3, capability is one
of the interesting concepts to study. Theorem 3.7 suggests to
consider the relationship between capability and varietal
capability. More precisely we may ask under what conditions
capability implies varietal capability or vice versa? This section
answers the above question in part and show the two concepts does
not coincide in general. Having a review of what has been done,
Burns and Ellis (1998), after introducing the concept of
$c$-capability, showed that every $c+1$-capable group is $c$-capable
group and hence is a capable group, but they construct a $2$-group
which is capable but is not $2$-capable. This example shows that
even in class of $p$-groups the capability does not imply
$c$-capability. However as an interesting fact they proved for
finitely generated abelian groups capability and $c$-capability are
equivalent. Now we concentrate on $\mathfrak{S}_{\ell}$, the variety
of solvable groups of length at most $\ell$ and prove the following
theorem.

\begin{thm}
Let $\mathfrak{S}_{\ell}$ be the variety of solvable groups of
length at most $\ell$. Then every $\mathfrak{S}_{\ell}$-capable
group is $\mathfrak{S}_{\ell-1}$-capable.
\end{thm}

\begin{proof}
Using Theorem 1.12 we have
$$\frac{{\mathfrak{S}_{\ell}}^*(G)}{C_G(\mathfrak{S}_{\ell-1}(G))}=\frac{G}{C_G(\mathfrak{S}_{\ell-1}(G))}.$$
Now the result follows immediately.
\end{proof}

An immediate consequence of the above theorem is that every
$\mathfrak{S}_{\ell}$-capable group is capable.

Comparing with $c$-capability, there is no difference in results.
The next theorem shows the converse of Theorem 4.1 is not true in
general, just the same as the result of Burns and Ellis. But the
difference is that here the counter example is in the class of
finitely generated abelian groups, exactly where the notions of
capability and $c$-capability coincide.

\begin{thm}
Let $n$ be a natural number, then the group $\mathbf{Z}_n\oplus
\mathbf{Z}_n$ is capable but it is not $\mathfrak{S}_2$-capable
\end{thm}

Finally we consider
$[\mathfrak{N}_{c_1},\mathfrak{N}_{c_2}]$-capability. This can be
considered as an special case of
$[\mathfrak{V},\mathfrak{W}]$-capability and one may suggest dealing
with the relation between $\mathfrak{V}$-capability,
$\mathfrak{W}$-capability and
$[\mathfrak{V},\mathfrak{W}]$-capability. Here, there can not be
explained more about that situation except the following theorem
which has a proof similar to that of Theorem 4.1.

\begin{thm}
Let $\mathfrak{V}$ be any variety then every
$[\mathfrak{V},\mathfrak{V}]$-capable group is
$\mathfrak{V}$-capable group.
\end{thm}
As before stated the converse of the above theorem is not true in
general, but in class of finitely generated abelian groups we have
the following theorem.

\begin{thm}
Let $G\cong \mathbb{Z}^{(k)}\oplus
\mathbb{Z}_{n_1}\oplus\cdots\oplus \mathbb{Z}_{n_t}$ be a finitely
generated abelian group, where $n_{i+1}\mid n_i$ for all $1\leq i
\leq t-1$. Then the following are equivalent:

\hspace{-.75cm}
$(i) \ G$ is capable.\\
$(ii) \ G$ is $\mathfrak{N}_c$-capable for some $c\geq 1$.\\
$(iii) \ G$ is $\mathfrak{N}_c$-capable for all $c\geq 1$.\\
$(iv) \ G$ is $[\mathfrak{N}_{c_1},\mathfrak{N}_{c_2}]$-capable for
all $c_1$, $c_2$ with $c_2<c_1\leq 2c_2$ or $c_1=c_2>1$.\\
$(v) \ k\geq 2$, or $k=0$, $t\geq 2$, and $n_1=n_2$.\\
\end{thm}

\end{document}